\documentclass[12pt,leqno,draft]{amsart}
\usepackage{amssymb,amsmath,amsthm}
\usepackage{graphicx}
\usepackage{color}
 \definecolor{darkblue}{RGB}{0,0,160}
 \usepackage[colorinlistoftodos,prependcaption,textsize=small]{todonotes}
\usepackage[lining]{libertine}   
\usepackage{cabin}
\usepackage[libertine]{newtxmath}
\usepackage[colorlinks=true,allcolors=darkblue]{hyperref}

\usepackage[T1]{fontenc}

\def\eps{\varepsilon}

\def\d{{\rm d}}

\def\R {\mathbb{R}}

\def\M {{\mathrm M}}

\def\1 {{\mbox{\boldmath 1}}}

\def \l {\langle}

\def \r {\rangle}

\def \and{\quad\text{and}\quad}

\newcommand{\cic}[1]{\mbox{\boldmath$#1$}}

\def \no#1#2#3 {{\bf #1} (#3), #2.}
\def \eds#1#2#3 {#1, #2, #3.}

\newcounter{thms}

\numberwithin{other}{section}

\newtheorem{theorem}[thms]{Theorem}
\newtheorem*{theorem*}{Theorem}
\newtheorem*{proposition*}{Proposition}

\numberwithin{corollary}{thms}
\newtheorem{lemma}[subsection]{Lemma}
\theoremstyle{definition}
\newtheorem{remark}[subsection]{Remark}
\numberwithin{equation}{section}

\title[Sparse estimate for vector-valued maximal functions]{A sparse estimate for multisublinear forms\\ involving vector-valued maximal functions}
 \author[A.\ Culiuc]{Amalia Culiuc}
 \address{\noindent School of Mathematics, Georgia Institute of Technology, \newline \indent Atlanta, GA 30332, USA}
\email{amalia@math.gatech.edu}
 \author[F. Di Plinio]{Francesco Di Plinio} \address{\noindent Department of Mathematics, University of Virginia,  \newline \indent Kerchof Hall,  Box 400137, Charlottesville, VA 22904-4137, USA   }
 \email{francesco.diplinio@virginia.edu}
 \author[Y.\ Ou]{Yumeng Ou}
 \address{\noindent Department of Mathematics, Massachusetts Institute of Technology, \newline \indent  77 Massachusetts Avenue, Cambridge, MA 02139, USA  }
\email{yumengou@mit.edu} 


 \thanks{F. Di Plinio was partially
supported by the National Science Foundation under the grants NSF-DMS-1500449 and  NSF-DMS-1650810,  by the Severo Ochoa Program SEV-2013-0323 and by Basque Government BERC Program 2014-2017.}

\begin{document}
\begin{abstract} We prove a sparse bound for the {{$m$-sublinear}} form associated to vector-valued maximal functions of Fefferman-Stein type. As a consequence, we show that the sparse bounds of multisublinear operators are  preserved via  $\ell^r$-valued extension. This observation is in turn used to deduce  vector-valued, multilinear weighted norm inequalities for multisublinear operators obeying sparse bounds, which are out of reach for the extrapolation theory developed by Cruz-Uribe and Martell in \cite{CUM}. As an example, vector-valued multilinear weighted inequalities for bilinear Hilbert transforms are deduced from the scalar sparse domination theorem of \cite{CuDPOu}.

\end{abstract}
\maketitle
\section{Main results} 
{{Let $m=1,2,\ldots$ and $\vec{p}=(p_1,\ldots,p_m)\in (0,\infty]^m$ be a generic $m$-tuple of exponents. This note is centered around the vector-valued $m$-sublinear maximal function
\begin{equation}
\label{FS}
\M_{\vec p,r} (\cic{f}^1,\ldots, \cic{f}^m)   := \left\|\sup_{Q} \prod_{j=1}^m \l f^j_k \r_{p_j,Q} \cic{1}_{Q}  \right\|_{\ell^ r(\mathbb C^N)}, \qquad 1\leq r\leq \infty.
\end{equation}
Here each $\cic f^j=(f^j_{1},\ldots,f^j_N)$ is a $\mathbb C^N$-valued locally $p_j$-integrable function on $\R^d$, the supremum is being taken over all cubes $  Q\subset \R^d$. and we have adopted for $\cic{a}=(a_1,\ldots,a_N) \in \mathbb C^N$ the usual notation 
\[
\|\cic{a}\|_{\ell^r} := 
\left( \sum_{k=1}^N |a_k|^r\right)^{\frac1r}, \quad0<r <\infty,\qquad \|\cic{a}\|_{\ell^\infty}:=\sup_{k=1,\ldots,N}|a_k|,
\]as well as
\[
\l f \r_{p,Q}:=  \frac{\| f\cic{1}_Q\|_p}{|Q|^{\frac1p}}.
\]The parameter $N$ is merely formal and all $\ell^r$-valued estimates below are meant to be independent of $N$ without explicit mention. Note that when $m=1$, \eqref{FS} reduces to the well studied Fefferman-Stein maximal function \cite[Ch.\ II.1]{Stein}. In fact,
it follows by H\"older's inequality that 
\begin{equation}
\label{FS3}
\M_{\vec p,r} (\cic{f}^1,\ldots, \cic{f}^m) \leq \prod_{j=1}^m \M_{p_j,r_j} (\cic{f}^j).
\end{equation} 
Therefore, the  full range of strong Lebesgue space estimates
\begin{equation}
\label{FS2}
\M_{\vec p,r} :\prod_{j=1}^{m} L^{q_j}(\R^d;\ell^{r_j})  \to L^{q}(\R^d), \quad   q= \frac{1}{ \sum_{j=1}^m\frac{1}{q_j}}, \quad r= \frac{1}{  \sum_{j=1}^m\frac{1}{r_j}}, \quad  1\leq p_j < \min\{ r_j,q_j\}
\end{equation}
and the weak-type endpoint
\begin{equation}\label{FS4}
\M_{\vec p,r} :\prod_{j=1}^{m} L^{p_j}(\R^d;\ell^{r_j}) \to L^{p,\infty}(\R^d),\quad p=\frac{1}{ \sum_{j=1}^m\frac{1}{p_j}}, \quad r= \frac{1}{  \sum_{j=1}^m\frac{1}{r_j}}, \quad  1\leq p_j < r_j
\end{equation}
are subsumed by the $m=1$ case  discussed in \cite[Ch.\ II.1]{Stein}, via H\"older's inequality in strong and weak-type spaces respectively.
Moreover, (\ref{FS3}) can be strengthened to the following form: given any partition $\mathcal{I}:=\{I_1,\ldots,I_s\}$ of $\{1,\ldots,m\}$, there holds
\[
\M_{\vec p,r}(\cic{f}^1,\ldots,\cic{f}^m)\leq \prod_{i=1}^s\M_{\vec{p}_i,r_i}(\cic{f}^{(i)})\leq \prod_{j=1}^m \M_{p_j,r_j} (\cic{f}^j),
\]where $\cic{f}^{(i)}:=\{\cic{f}^j\}_{j\in I_i}$, $\vec{p}_i:=(p_j)_{j\in I_i}$, and $1/r_i:=\sum_{j\in I_i}1/r_j$.

The first main result of this note, Theorem \ref{main:th} below, is a nearly sharp  sparse estimate involving vector-valued $m$-sublinear maximal functions of the form
\begin{equation}\label{form:eq}
\int_{\R^d}   \prod_{i=1}^s\M_{\vec p_i,r_i}( \cic{f}^{(i)})(x) \, \d x,
\end{equation}
which strengthens the Lebesgue space estimates \eqref{FS2}, \eqref{FS4}.
 As an application of Theorem \ref{main:th}, we obtain a structural result on sparse bounds,  Theorem \ref{sparse:th} below, which seems to have gone unnoticed in previous literature: sparse bounds in the scalar setting self-improve to the $\ell^r$-valued setting. In other words, if a given  sequence of operators are known to obey a uniform sparse bound,  the vector-valued operator associated to the sequence satisfies the same $\ell^r$-valued sparse bound, without the need for additional structure of the operators. 

We proceed to define the notion of sparse bound we have referred hitherto. A countable collection $\mathcal Q$ of cubes of $\R^d$ is \emph{sparse} if there exist a pairwise disjoint collection of sets $\{E_Q:\,Q\in\mathcal{Q}\}$ such that for each $Q \in \mathcal Q$ there holds \[
E_Q\subset Q, \qquad |E_Q|>\frac{1}{2}|Q|.\]   Let $n\geq 1$ and $T$ be a $n$-sublinear operator mapping ($n$ copies of)  $L^\infty_0(\R^d;\mathbb C)$ into locally integrable functions. If $\vec p\in (0,\infty)^{n+1}$,  the sparse $\vec p$ norm of $T$, denoted by $\|T\|_{\vec p}$, is the least constant $C>0$ such that for all $(n+1)$-tuples  $ \vec{ {g}}=(g^1,\ldots,g^{n+1})\in L^\infty_0(\R^d;\mathbb C)^{n+1} $ we may find a sparse collection $\mathcal Q=\mathcal Q(\vec{ {g}})$ such that
\[
\begin{split}
\left|\l T(g^1,\ldots,g^{n}),g^{n+1}\r\right| \leq C  \sum_{Q\in \mathcal Q} |Q| \prod_{j=1}^{n+1} \l g^j \r_{p_j,Q}.\end{split}
\]
Beginning with the breakthrough work of Lerner \cite{Ler2013}, sparse bounds have recently come to prominence in the study of singular integral operators, both at the boundary of \cite{Bar2017,CoCuDPOu,Lac2015,Ler2016} and well beyond Calder\'on-Zygmund theory \cite{BFP,CoCuDPOu,CuDPOu,Lac2017}; the list of references provided herein is necessarily very far from being exhaustive. As we will see in Section \ref{S3}, their interest lies in that they imply rather easily quantitative weighted norm inequalities for the corresponding operators.  

The concept of sparse bound extends naturally    to vector-valued operators. If \[\cic{T}=\{T_1,\ldots,T_N \}\] is a sequence of  $n$-sublinear operators as above, we may let $\cic{T}$ act on $ L^\infty_0(\R^d;\mathbb C^N)^{n} $ as
\[
\l \cic{T}(\cic{f}^1,\ldots,\cic{f}^{n}), \cic{f}^{n+1}\r := \sum_{k=1}^N \l  {T}_k( {f}^1_k,\ldots, f^{n}_k),  {f}^{n+1}_k\r.
\]
Let $(r_1,\ldots,r_{n+1})$ be a Banach H\"older $(n+1)$-tuple, that is 
\begin{equation}
\label{Holder}
1\leq r_j\leq \infty,\quad j=1,\ldots,n+1,\qquad
r:=\frac{r_{n+1}}{r_{n+1}-1}=\frac{1}{\sum_{j=1}^{n} \frac{1}{r_{j}}}
\end{equation}
and define the sparse $(\vec p,\vec r)$-norm of $\cic{T}$ as the least constant $C>0$ such that 
\begin{equation}\label{vvdef:eq}
\begin{split}
 |\l  \cic{T}(\cic{f}^1,\ldots,\cic{f}^{n}), \cic{f}^{n+1}\r| \leq C   \sum_{Q\in \mathcal Q} |Q| \prod_{j=1}^{n+1} \left\langle \| \cic{f}^{j}\|_{\ell^{r_j}}
\right\rangle _{p_j,Q} 
\end{split}
\end{equation}
for all $(n+1)$-tuples $\vec {\cic{f}} \in L^\infty_0(\R^d; \mathbb C^N)^{n+1} $ and for a suitable choice of $\mathcal Q=\mathcal Q(\vec {\cic{f}})$. We denote such norm by $\|\cic{T} \|_{(\vec p,\vec r)  }$.  
Our punchline result is the following.
\begin{theorem}\label{sparse:th} Let $\vec p\in [1,\infty)^{n+1}$ and $\vec r$ be as in \eqref{Holder} with the assumption $r_j>p_j$. Then
 \begin{equation}
\label{sparse:eq}
\|\{T_1,\ldots,T_N\} \|_{( \vec{p},\vec r) } \lesssim   \sup_{k=1,\ldots, N}  \| {T}_k \|_{\vec p}.
\end{equation}
The implicit constant depends on the tuples $\vec p$ and $\vec r$ and on the dimension $d$.
\end{theorem}
\begin{remark}The recent preprint \cite{BM2} contains  a direct proof of $\ell^r$-valued sparse form estimates for   multilinear multipliers with singularity along one-dimensional subspaces, generalizing the paradigmatic bilinear Hilbert transform, as well as for the variation norm Carleson operator. Theorem \ref{sparse:th} thus allows to recover these results of \cite{BM2} from the corresponding scalar valued results previously obtained in \cite{CuDPOu}, which is recalled in \eqref{BHTsparse} below, and \cite{DPDU} respectively. 
\end{remark}

We refer the readers to Subsection \ref{vvfroms} for a proof of Theorem \ref{sparse:th} and proceed with introducing the main theorem concerning sparse bounds of multisublinear forms of type (\ref{form:eq}), whose proof is postponed to Section \ref{SecPf}. 
\begin{theorem}\label{main:th} Let there be given $m$-tuples $\vec p=(p_1,\ldots,p_m)\in [1,\infty)^{m}, \vec r=(r_1,\ldots, r_{m})\in [1,\infty]^m$ with     \[ \frac1r:=\sum_{j=1}^m \frac{1}{r_j}, \qquad 
 p_j<r_j, \; j=1,\ldots, {m}.    
\]
\noindent
\emph{1.} Let $\eps>0$. There exists a sparse collection $\mathcal Q$ such that
\begin{equation}\label{main:eq}
\int_{\R^d} \prod_{j=1}^m \M_{p_j,r_j}(\cic{f}^j)(x)\,\d x \lesssim \sum_{Q\in \mathcal Q} |Q| \prod_{j=1}^{m} \left\langle \| \cic{f}^{j}\|_{\ell^{r_j}}\right\rangle _{p_j+\eps,Q}.
\end{equation}
The implicit constant is allowed to depend on $\eps>0$, as well as the tuples $\vec p$, $(r_1,\ldots,r_{m})$ and on the dimension $d$.
\vskip1mm\noindent
\emph{2.} There exists a sparse collection $\mathcal Q$, possibly different from above, such that
\begin{equation}\label{main2:eq}
\int_{\R^d} \M_{\vec p,r}(\cic{f}^1,\ldots,\cic{f}^m)(x)\,\d x \lesssim \sum_{Q\in \mathcal{Q}} |Q| \prod_{j=1}^{m} \left\langle \| \cic{f}^{j}\|_{\ell^{r_j}}\right\rangle _{p_j,Q}.
\end{equation}
The implicit constant is allowed to depend on $\vec p$, $(r_1,\ldots,r_{m})$ and $d$.
\end{theorem}

\begin{remark}
An immediate consequence of Theorem \ref{main:th} is a   sparse bound for multisublinear forms involving any $\M_{\vec p,r}$. More precisely, for any partition $\mathcal{I}:=\{I_1,\ldots,I_s\}$ of $\{1,\ldots,m\}$, there exists a sparse collection $\mathcal Q$ (depending on $\mathcal{I}$) such that
\begin{equation}
\label{main3:eq}
\int_{\R^d}   \prod_{i=1}^s\M_{\vec p_i,r_i}( \cic{f}^{(i)})(x) \, \d x  \lesssim \sum_{Q\in \mathcal Q} |Q| \prod_{j=1}^{m} \left\langle \| \cic{f}^{j}\|_{\ell^{r_j}}\right\rangle _{p_j+\eps,Q}.
\end{equation}
We do point out that even though the $s=1$ case of (\ref{main3:eq}), i.e. when the partition $\mathcal{I}$ contains only $\{1,\ldots,m\}$ itself, already implies a sparse bound for the form on the left hand side of (\ref{main2:eq}), it fails to recover the full strength of (\ref{main2:eq}) due to the $\eps$-loss.
\end{remark}

\subsection{Vector valued sparse estimates from scalar ones}\label{vvfroms}
In this subsection we prove Theorem \ref{sparse:th}, with the key ingredients being (\ref{main2:eq}) and the following observation, which we record as a lemma; a similar statement may be found in the argument following \cite[Appendix A, (A.8)]{CuDPOu}.
\begin{lemma} \label{lemma1:lemma}Let   $\vec {\cic{f}} \in L^\infty_0(\R^d; \mathbb C^N)^{n+1} $. Then 
\begin{equation}
\label{lemma1:eq} 
|\l  \cic{T}(\cic{f}^1,\ldots,\cic{f}^{n}), \cic{f}^{n+1}\r|  \leq 2 \left(\sup_{k=1,\ldots, N}  \| {T}_k \|_{\vec p} \right) 
\int_{\R^d}  \M_{\vec p,1}( \cic{f}^{1},\ldots,\cic{f}^{n+1})\,\d x
\end{equation}
where $\vec p=(p_1,\ldots,p_{n+1})$.
\end{lemma}
\begin{proof} Normalize  $\| {T}_k \|_{\vec p}=1$ for $k=1,\ldots,N$. Using the definition, for $k=1,\ldots,N$ we may find sparse collections $\mathcal Q_1,\ldots,\mathcal Q_N$ such that
\[
|\l  {T}_k( {f}^1_k,\ldots, f^{n}_k),  {f}^{n+1}_k\r| \leq \sum_{Q_k \in   \mathcal Q_k } |Q_k| \prod_{j=1}^{n+1} 
\left\langle   {f}^{j}_k
\right\rangle _{p_j,Q_k} \leq 2\int_{\R^d} F_k(x)\,\d x, 
\]
having defined 
\[
F_k = \sum_{Q_k \in \mathcal Q_k} \left(\prod_{j=1}^{n+1} 
\left\langle   {f}^{j}_k
\right\rangle_{p_j,Q_k}\right)\cic{1}_{E_{Q_k}},
\]
where the last inequality follows from the pairwise disjointness of the distinguished major subsets $E_{Q_k}\subset Q_k$, with  $2|E_{Q_k}|\geq| Q_k|$. Therefore, 
\[
|\l \cic{T}(\cic{f}^1,\ldots,\cic{f}^{n}), \cic{f}^{n+1}\r| \leq  2 \int_{\R^d} \M_{\vec{p},1}(\cic{f}^1,\ldots,\cic{f}^{n+1})(x)\,\d x. 
\]
\end{proof}
Theorem \ref{sparse:th} then immediately follows from Lemma \ref{lemma1:lemma} recalling (\ref{main2:eq}).
\begin{remark}
Lemma \ref{lemma1:lemma} obviously applies to any $(n+1)$-sublinear form $\cic{\Lambda}(\cic{f}^1,\ldots,\cic{f}^{n+1})$, not necessarily of the form $\l \cic{T}(\cic{f}^1,\ldots,\cic{f}^{n}), \cic{f}^{n+1}\r$.  We then record the following observation: in the scalar valued case $N=1$, there holds the equivalence
\begin{equation}\label{equiv:eq}
\sup_{  \mathcal Q \, \mathrm{sparse}}\sum_{Q\in\mathcal{Q}}|Q|\prod_{j=1}^m\left\langle f^j\right\rangle_{p_j,Q}\sim \int_{\R^d} \M_{\vec{p}}(f^1,\ldots,f^m)(x)\,\d x.
\end{equation}
Incidentally, this is an alternative proof of the useful ``one form rules them all'' principle of Lacey and Mena Arias \cite[Lemma 4.7]{LMena}.
Indeed, (\ref{equiv:eq}) follows from applying Lemma \ref{lemma1:lemma} to the case $N=1$ and to the $m$-sublinear form on the left hand side of (\ref{equiv:eq}). Such an  equivalence does not seem to hold in the vector-valued case.\end{remark}

}}

\section{Proof of {{Theorem}} \ref{main:th}} \label{SecPf}
{{The proof of the main result is iterative in nature and borrows some of the ingredients from the related articles \cite{CoCuDPOu,DPHL}. Throughout, we assume that the tuples $\vec p=(p_1,\ldots,p_m)$ and $(r_1,\ldots,r_{m})$ as in the statement of Theorem \ref{main:th} are fixed. We first prove part 1, and the proof of part 2, which is very similar and is in fact simpler, will be given at the end of the section.

\subsection{Truncations and a simple lemma} 
We start by defining suitable truncated versions of the Fefferman-Stein maximal functions \eqref{FS}. For $s,t>0$, write
\[
\mathsf{A}_j^{s,t} {\cic f^j}  := 
\left\|\sup_{s<\ell(Q)\leq t} \l f^j_k \r_{p_j,Q} \cic{1}_{Q}  \right\|_{\ell^ {r_j}(\mathbb C^N)}, \qquad j=1,\ldots,m.
\]
Note that $\forall j$,
\begin{equation}
\label{formal:eq}   \sup_{s<t}\mathsf{A}_j^{s,t}{\cic{f}^j} = \M_{p_j,r_j} {\cic{f}^j}.
\end{equation}
We will be using the following key lemma, which is simply the lower semicontinuity property of truncated maximal operators.
\begin{lemma} \label{trunc:lemma}Let $x,x_0\in \R^d$ and $s\gtrsim\mathrm{dist}(x_0,x)$. Then
\[
\mathsf{A}_j^{s,t} {\cic{f}^j} (x) \lesssim  \mathsf{A}_j^{s,t} {\cic{f}^j} (x_0).
\]
\end{lemma} 

\subsection{Main argument} We work with a fixed $\delta>0$; we will let $\delta\to 0 $  in the limiting argument appearing  below.
For a cube $Q$ we  define  further localized versions as
\begin{equation}
\label{defaq}
\mathsf{A}_j^{Q} (\cic{f}^j):=  \cic{1}_Q \mathsf{A}_j^{\delta,\ell(Q)} (\cic{f}^j) =  \cic{1}_Q \mathsf{A}_j^{\delta,\ell(Q)}({\cic{f}^j} \cic{1}_{3Q})\end{equation}
where the the last inequality follows from support consideration.
%

By standard  limiting and  translation invariance arguments,   \eqref{main:eq} is reduced to the following sparse estimate: if $Q$ is a cube belonging to one of the $3^d$ standard dyadic grids, then
\begin{equation}
\label{proof:eq:1}
\Lambda_{Q}(\cic{f}^1,\ldots,\cic{f}^{m}):= \int_{Q} \prod_{j=1}^m \mathsf{A}_j^{Q} ({\cic{f}^j})   (x)\,\d x   \lesssim \sum_{L\in \mathcal Q} |L| \prod_{j=1}^{m} \left\langle \| \cic{f}^{j}\|_{\ell^{r_j}}
\right\rangle _{p_j+\eps,L}
\end{equation}
uniformly over $\delta>0$, where $\mathcal{Q}$ is a stopping collection of pairwise disjoint cubes. Estimate \eqref{proof:eq:1} follows by iteration of the following  lemma: the iteration procedure is identical to the one used, for instance, in the proof of  \cite[Theorem 3.1]{Ler2016} and is therefore omitted.
\begin{lemma}\label{iter:lemma} There exists a constant $\Theta$, uniform in the data below, such that the following holds.  Let $Q$ be a dyadic cube and $ ( \cic{f}^1,\ldots,\cic{f}^{m})\in L^\infty_0(\R^d;\mathbb C^N)^{m}.$ Then there exists a collection $L\in \mathcal Q$  of pairwise disjoint dyadic subcubes of $Q$ such that
\[
\sum_{L\in \mathcal Q} |L| \leq 2^{-16}|Q|
\]
and
\[
\Lambda_{Q}(\cic{f}^1,\ldots,\cic{f}^{m}) \leq \Theta |Q| \prod_{j=1}^{m} \left\langle \| \cic{f}^{j}\|_{\ell^{r_j}}
\right\rangle _{3Q,p_j+\eps}+ \sum_{\substack{L \in \mathcal Q} } \Lambda_{L}(\cic{f}^1,\ldots,\cic{f}^{m}).
\]
\end{lemma}
\subsection{Proof of Lemma \ref{iter:lemma}} We can assume everything is supported in $3Q$. By horizontal dilation invariance we may assume $|Q|=1$. By vertical scaling we may assume $  \langle \| \cic{f}^{j}\|_{\ell^{r_j}}
 \rangle_{p_j+\eps,3Q}=1$ for all $j=1,\ldots, m$. 
Define the  collection $L\in\mathcal Q$ as the maximal dyadic cubes of $\R^d$ such that $9L\subset E_Q$ where
\[
E_Q =\bigcup_{j=1}^m\left\{ x\in Q: \M\circ{\mathsf A}_j^{Q}({\cic{f}^j}) (x) \geq C\right\}, 
\]
here $\M$ is the usual Hardy-Littlewood maximal function. If $C$ is large enough, using  the Lebesgue space boundedness of $\M \circ{\mathsf A}_j^{Q}$ with the choices $q_j=p_j+\eps$ in \eqref{FS2}, the set $E_Q$ has small measure compared to $Q$ and same for the pairwise disjoint cubes $L$ in the stopping collection $\mathcal Q$.

As a consequence of the construction of $\mathcal Q$ and of Lemma \ref{trunc:lemma} we obtain the following properties for all $j=1,\ldots,m$ and $L \in \mathcal Q$ 
\begin{align}
\label{proof:eq:11} &  \sup_{x \not \in E_Q}{\mathsf A}_j^{Q}({\cic{f}^j}) (x)\lesssim 1,\\
\label{proof:eq:12} &\sup_{L'\gtrsim L} \l {\mathsf A}_j^{Q}({\cic{f}^j}) \r_{1,L'}\lesssim 1,\\
\label{proof:eq:13} &\sup_{x\in L}  \mathsf{A}_j^{\ell(L),\ell(Q)}( {\cic{f}^j}) (x) \lesssim 1.
\end{align}
The third property follows from the fact that if $x\in L$ there is a point $x_0\in L'$, with $L'$ a moderate dilate of $L$, with small $\M_j$, so that one may apply Lemma \ref{trunc:lemma}.

We now prove the main estimate.  By virtue of \eqref{proof:eq:11}, 
\begin{equation}
\label{fs9}
\int_{Q\setminus E_Q}   \prod_{j=1}^m\mathsf{A}_j^{Q} ({\cic{f}^j})   (x) \,\d x \lesssim   1.
\end{equation}
Given that
 ${L \in \mathcal Q}$ cover $E_Q$ and  are pairwise disjoint  
it then suffices to prove that for each $L$
\begin{equation}
\label{proof:eq:2}
  \int_{L}  \prod_{j=1}^m \mathsf{A}_j^{Q} ({\cic{f}^j})   (x) \,\d x   \leq \Lambda_L(\cic{f}^1,\ldots,\cic{f}^{m})+  C|L|
\end{equation}
and sum this estimate up.
Observe that the left hand side of \eqref{proof:eq:2} is bounded by the sum
\begin{equation}
\label{2}
\int_{L}\prod_{j=1}^m\mathsf{A}_j^{\delta,\ell(L)} {\cic{f}^j}(x)\,\d x+ \sum_{\tau_1,\ldots,\tau_m}\int_{L}\prod_{j=1}^m\mathsf{A}_j^{\tau_j}\cic{f}^j (x)\,\d x,
\end{equation}
where $\mathsf{A}_j^{\tau_j}$ is either $\mathsf{A}_j^{\delta,\ell(L)}$ or $\mathsf{A}_j^{\ell(L),\ell(Q)}$, and the sum is over all the possible combinations of $\{\tau_1,\ldots,\tau_m\}$ except the one with $\mathsf{A}_j^{\delta,\ell(L)}$ appearing for all $j$. Note that the first term in the above display is equal to $\Lambda_L(\cic{f}^1,\ldots,\cic{f}^{m})$, so it suffices to show that
\begin{equation}
\label{fs10}
\sum_{\tau_1,\ldots,\tau_m}\int_{L}\prod_{j=1}^m\mathsf{A}_j^{\rho_j}\cic{f}^j (x)\,\d x\lesssim |L|
\end{equation}
where $\mathsf{A}_j^{\rho_j}$ is either $\mathsf{A}_j^{Q}$ or $\mathsf{A}_j^{\ell(L),\ell(Q)}$ and $\mathsf{A}_j^{\ell(L),\ell(Q)}$ appears at least at one $j$.
This is because the left hand side is larger than the second term of \eqref{2}. But this is immediate by using the $L^1$ estimate of \eqref{proof:eq:12} on the terms of the type $\mathsf{A}_j^{Q} {\cic{f}^j}$ and the
  $L^\infty$ estimate of \eqref{proof:eq:13} on the terms $\mathsf{A}_j^{\ell(L),\ell(Q)} {\cic{f}^j}$ respectively. The proof is complete.
  
\subsection{Proof of (\ref{main2:eq})}
The proof of (\ref{main2:eq}) proceeds very similarly to the one given above. Write $\vec{\cic f}=(\cic{f}^1,\ldots,\cic{f}^m)$ for simplicity and define the multilinear version of the truncated operator
\[
\mathsf{A}^{s,t} {\vec{\cic f}}  := 
\left\|\sup_{s<\ell(Q)\leq t} \prod_{j=1}^m \l f^j_k \r_{p_j,Q} \cic{1}_{Q}  \right\|_{\ell^ {r}(\mathbb C^N)}, \quad s,t>0.
\]With this definition of  $\mathsf{A}^{s,t}$, the analogues of (\ref{formal:eq}) and Lemma \ref{trunc:lemma} still hold. Therefore, a similar liming argument as above reduces the matter to showing 
\[
\Lambda_{Q}(\vec{\cic f}):=
\int_{Q} \mathsf{A}^{Q} \vec{\cic f}(x)\,\d x\lesssim \sum_{L\in\mathcal{Q}}|L|\prod_{j=1}^m\left\langle \|\cic{f}^j\|_{r_j}\right\rangle_{p_j,L}
\]uniformly over $\delta>0$ for some stopping collection $\mathcal{Q}$, where $\mathsf{A}^Q$ is the localized version of $\mathsf{A}^{s,t}$ defined as in \eqref{defaq}.
The proof of the last display proceeds by iteration of the analogous result to  Lemma \ref{iter:lemma}:
for any dyadic cube $Q$  and $\vec{\cic{f}} \in L^\infty_0(\R^d;\mathbb C^N)^{m}$   there exists a collection $L\in \mathcal Q$  of pairwise disjoint dyadic subcubes of $Q$ such that
\[
\sum_{L\in \mathcal Q} |L| \leq 2^{-16}|Q|
\]
and
\[
\Lambda_{Q}(\vec{\cic f}) \leq \Theta |Q| \prod_{j=1}^{m} \left\langle \| \cic{f}^{j}\|_{\ell^{r_j}}
\right\rangle _{3Q,p_j}+ \sum_{\substack{L \in \mathcal Q} } \Lambda_{L}(\vec{\cic f}).
\]
 To prove the last claim, the following changes are needed in the proof of Lemma \ref{iter:lemma}. We use instead the normalization $\left\langle\|\cic{f}^j\|_{\ell^{r_j}}\right\rangle_{p_j,3Q}=1$ without the $\varepsilon$, and define the exceptional set without the extra Hardy-Littlewood maximal function, i.e.
\[
E_Q:=\{x\in Q:\, \mathsf{A}^Q(\vec{\cic{f}})(x)\geq C\}.
\]Since, from \eqref{FS4}, $\mathsf{A}^Q$  has the weak-type bound at $\prod_{j=1}^m L^{p_j}$, the measure of $E_Q$ is small for sufficiently large $C$. Note that one still has analogues of estimates (\ref{proof:eq:11}) and (\ref{proof:eq:13}) for $\mathsf{A}^{Q} \vec{\cic f}$ in place of ${\mathsf A}_j^{Q}({\cic{f}^j})$, and (\ref{proof:eq:12}) becomes irrelevant in this case. The proof is completed by using these estimates as in \eqref{fs9} and \eqref{fs10} respectively.
  }}

\section{Vector-valued weighted norm inequalities}\label{S3}
 
Using the almost equivalence between scalar and vector-valued sparse estimates of Theorem \ref{sparse:th}, we prove vector-valued weighted norm inequalities for  $n$-sublinear operators with controlled sparse $\vec p=(p_1,\ldots,p_{n+1})$ norm. The weighted bounds can be obtained via estimates for the form
\[
(g^1,\ldots, g^{n+1}) \mapsto  \mathrm{P}_{\vec p}(g^1,\ldots, g^{n+1}; F) :=\int_{F} \M_{(p_1,\ldots p_n)} (g^1,\ldots,g^{n}) (x)
\M_{p_{n+1}} g^{n+1} (x) \, \d x.
\]
where
\begin{equation}
\label{FSscal}
\M_{\vec t} (g^1,\ldots, g^n)   := \left|\sup_{Q} \prod_{j=1}^n \l g^j \r_{t_j,Q} \cic{1}_{Q}\right|
\end{equation}
is the scalar valued version of \eqref{FS}.
We consider H\"older tuples
\begin{equation}
\label{holder2}
1\leq q_1,\ldots, q_n\leq \infty, \qquad q : =\frac{1}{ \sum_{j=1}^n \frac{1}{q_j}} \leq 1
\end{equation}
and 
 weight vectors $\vec v=(v_1,\ldots,v_n)$ in $\R^d$  with 
\begin{equation}
\label{holder3}
v= \prod_{j=1}^n v_j^{\frac{q}{q_j}}.
\end{equation}
It is well known \cite[Theorem 3.3]{LOP+} that
\begin{equation}
\label{weighted:eq:1}
\M_{(p_1,\ldots p_n)}: \prod_{j=1}^n L^{q_j}(v_j  ) \to  L^{q}(v  )  \iff q_1> p_1,\ldots, q_n> p_n, \quad [\vec v]_{A_{(q_1,\ldots,q_n)}^{(p_1,\ldots,p_n,1)}}< \infty\end{equation}
where the vector  weight characteristic appearing above is defined more generally by
\begin{equation}
\label{mwc}
[\vec v]_{A_{(q_1,\ldots,q_n)}^{(t_1,\ldots,t_{n+1})}}
:= \sup_{Q}\left( \l v \r_{\frac{t_{n+1}}{q-(q-1)t_{n+1}},Q}^{\frac{1}{q}} \prod_{j=1}^n \l (v_{j})^{-1}\r_{\frac{t_j}{q_j-t_j},Q}^{\frac{1}{q_j}}  \right)
<\infty.
\end{equation}
When $n=1$, the above characteristics generalize the familiar $A_t$ (Muckenhoupt) and $RH_t$ (Reverse H\"older) classes, namely
\[
A_{q}^{(t_1,t_2)} = A_{\frac{q}{t_1}} \cap RH_{\frac{t_2}{q-(q-1)t_2}}.
\]  
\begin{theorem} \label{weighted:thm}
Let  $(q_1,\ldots,q_n), q $ be as in \eqref{holder2} and let $\vec v=(v_1,\ldots,v_n), v$ be as in \eqref{holder3}. Assume  that 
\begin{itemize}
\item[1.]   $  \displaystyle\sup_{j=1,\ldots,N}\displaystyle \|T_j\|_{\vec p} \leq 1$  for some $\vec p=(p_1,\ldots,p_{n+1})$ with $1\leq p_1\leq q_1,\ldots,1\leq p_n\leq q_n$;
\item[2.] condition \eqref{weighted:eq:1} holds, namely   $$\vec v \in A_{(q_1,\ldots, q_n)}^{(p_1,\ldots,p_n,1)};$$
 \item[3.] there exists $t\in [1, p_{n+1}']$ such that
\[
v\in A_t \cap RH_{\frac{p_{n+1}}{t(1-p_{n+1})+p_{n+1}}}.
\] 
\end{itemize}
  Then the vector-valued strong type bound
  \begin{equation}
\label{stb}
\cic{T}: \prod_{j=1}^n L^{q_j}(v_j ; \ell^{r_j} ) \to  L^{q}(v; \ell^{r})  
\end{equation}
  holds true whenever $r_1\geq p_1, \ldots, r_n\geq p_n, r_{n+1}=r' \geq p_{n+1}$.
 \end{theorem}

\begin{proof} As 
$ \|T_j\|_{\vec p} \leq 1$ for all $j$, Theorem \ref{sparse:th} implies that there exists a sparse collection $\mathcal Q$ such that \begin{equation}
\label{AUX:eq1}
|\l \cic{T}(\cic{f}^1,\ldots,\cic{f}^{n}), \cic{f}^{n+1}\r|\lesssim   
\sum_{Q\in \mathcal Q } |Q|\prod_{j=1}^{n+1}\left\langle\|\cic{f}^j\|_{\ell_{r_j}} \right\rangle_{p_j,Q} 
\end{equation}
under the the assumptions $ r_j>p_j, j=1,\ldots, n+1.$
By interpolation, it suffices to prove the weak-type analogue of \eqref{stb}.
 We use the well known principle \begin{equation}
\label{wti}
\left\|\cic{T}: \prod_{j=1}^n L^{q_j}(v_j ; \ell^{r_j} ) \to  L^{q,\infty}(v; \ell^{r})\right\|\lesssim \sup
 \inf_{\substack{G\subset F\\ v(F)
\leq 2v(G) } } \frac{|\l \cic{T}(\cic{f}^1,\ldots,\cic{f}^{n}), \cic{f}^{n+1}v\cic{1}_G\r|}{v(F)^{1-\frac1q}}, 
\end{equation}
where the supremum is taken over sets $F\subset \R^d$ of finite measure, $\cic{f}^j \in L^{q_j}(v_j ; \ell^{r_j} ), j=1,\ldots, n$  of unit norm, and functions $\cic{f}^{n+1}$ with
$\|\cic{f}^{n+1}\|_{L^\infty(\R^d;\ell^{r_{n+1}})}\leq 1$. Fix $F,\cic{f}^{j} $ as such and
introduce the scalar-valued functions $g^j:=\|\cic{f}^{j}\|_{\ell_{r_{j}}}$, $j=1,\ldots,n+1$.
Set,
\[
E=\left\{x\in \R^d: \mathrm{M}_{(p_1,\ldots,p_n)}(g^1,\ldots,g^n) > \beta^{\frac1q} v(F)^{-\frac{1}{q}} \right\},
\]
 where $\beta>0$ will be determined at the end. 
We let $\widetilde G=\R^d\setminus E$ and finally we define the smaller set $G=F\setminus E'$ where $E'$ is the union of the  maximal dyadic cubes $Q$ such that $|Q|\leq 2^{5}|Q\cap E|$. Notice that 
\[
|E'| \leq 2^{5} |E| \implies v(E') \leq C([v]_{A_\infty}) v(E) < \frac{C}{\beta} v(F) \leq \frac{1}{2} v(F)
\]  
by choosing $\beta$ large enough and relying upon the bound \eqref{weighted:eq:1} to estimate $v(E)$. Therefore $G$ is a major subset of $F$. In this estimate we have used that $v\in A_\infty$, which is guaranteed by the third assumption of the theorem. 

Now, the argument used in \cite[Appendix A]{CuDPOu} applied to \eqref{AUX:eq1} with $\cic{f}^{n+1}$ replaced by $\cic{f}^{n+1}v\cic{1}_G$ returns
\begin{equation}
\begin{split}
\label{aux:eq2} |\l \cic{T}(\cic{f}^1,\ldots,\cic{f}^{n}), \cic{f}^{n+1}v\cic{1}_G\r| & \lesssim \sum_{\substack{Q\in \mathcal Q\\ |Q\cap \widetilde G| \geq 2^{-5}|Q|} } |Q|\left(\prod_{j=1}^{n} \langle g ^j \rangle_{p_j,Q}\right) \langle g^{n+1}v\cic{1}_F \rangle_{p_{n+1},Q} \\ & \lesssim 
\mathrm{P}_{\vec p }(g^1,\ldots, g^n, g^{n+1} v\cic{1}_{F};\R^d\backslash E) . 
\end{split}
\end{equation}
Further,
if $ t $ is as in the third assumption, an interpolation argument between \eqref{weighted:eq:1} and the $L^\infty$ estimate off the set $E$ yields  
\[
\big\|\mathrm{M}_{(p_1,\ldots,p_n)}(g^1,\ldots,g^n)\cic{1}_{\R^d\backslash E}\big\|_{L^{t}(v)} \lesssim v(F)^{\frac1t-\frac{1}{q}}.
\]
Therefore 
\[
\begin{split}
 &\quad  |\l \cic{T}(\cic{f}^1,\ldots,\cic{f}^{n}), \cic{f}^{n+1}v\cic{1}_G\r|\lesssim \mathrm{P}_{\vec p}(g^1,\ldots, g^n, g^{n+1} v\cic{1}_{F};\R^d\backslash E)  \\ &=
\int_{\widetilde G}\left( \mathrm{M}_{(p_1,\ldots,p_n)}(g^1,\ldots,g^n) v^{\frac1t}\right)\left(\mathrm{M}_{p_{n+1}}(g^{n+1} v\cic{1}_F) v^{-\frac1t}\right) \, \d x
\\  &\leq \big\|\mathrm{M}_{(p_1,\ldots,p_n)}(g^1,\ldots,g^n)\cic{1}_{\R^d\backslash E}\big\|_{L^{t}(v)}\|\mathrm{M}_{p_{n+1}}( v\cic{1}_F)\|_{L^{t'}(v^{1-t'})} 
\lesssim v(F)^{\frac1t-\frac{1}{q}} v(F)^{\frac{1}{t'}}=  v(F)^{1-\frac{1}{q}}
\end{split}
 \]
which, combined with \eqref{wti}, gives the desired result. Note that the third assumption, which is equivalent \cite{IMS} to   \[
v^{1-t'} \in A_{\frac{t'}{p_{n+1}} }
\] 
was used to  ensure the boundedness of $\M_{p_{n+1}}$ on $L^{t'}(v^{1-t'})$.  The proof is thus completed. 
\end{proof}
\begin{remark}\label{qm1}  Theorem \ref{weighted:thm} does not cover the range $q>1$.  In that range, in fact, \eqref{stb} continues to hold with   conditions 2.\ and 3.\ of Theorem \ref{weighted:thm}     replaced by a single   condition of multilinear type. To wit, if  $\|\{T_1,\ldots,T_N\}\|_{\vec p}<\infty$ with
\[  1\leq p_1\leq \min\{ q_1,r_1\},\ldots,1\leq p_n\leq \min\{ q_n,r_n\},\qquad  1\leq p_{n+1}\leq \min\left\{\frac{q}{q-1}, r_{n+1}\right\}\]
and $\vec v \in A^{(p_1,\ldots,p_{n+1})}_{(q_1,\ldots,q_n)}$, then the bound \eqref{stb} holds true. The proof uses the sparse bound \eqref{AUX:eq1} in exactly the same fashion as \cite[Theorem 3]{CuDPOu}. When $q\leq 1$, we are not aware of a fully multilinear sufficient condition on the weights leading to    estimate \eqref{stb}; Theorem \ref{weighted:thm} is a partial substitute in this context.
\end{remark}
\begin{remark} As the multilinear weighted classes \eqref{mwc} are not amenable to (restricted range) extrapolation, Theorem \ref{weighted:thm}, as well as its corollaries described in the next section, cannot be obtained within the multilinear extrapolation theory developed in the recent article \cite{CUM}.
\end{remark}
 
 \subsection{An example: the bilinear Hilbert transform}   We show how, in view of the scalar sparse domination results of \cite{CuDPOu}, Theorem \ref{weighted:thm} applies to a class of operators which includes the bilinear Hilbert transform. Let $T_{{m}}$ be bilinear operators whose action on Schwarz functions is given by
 \begin{equation}
\label{IN1}
\l T_m(g^1,g^2),g^3\r= \int \displaylimits_{\xi_1+\xi_2+\xi_3=0}  m(\xi) 
\prod_{j=1}^3 \widehat {g^j}(\xi_j) \, \d \xi .
\end{equation}
Here $m$ belongs to the class $\mathcal M$  of  bilinear  Fourier multipliers with singularity along the one dimensional subspace $\{\xi \in \R^3: \xi_1=\xi_2\}$; that  is
\begin{equation}
\label{decay}
\sup_{m\in \mathcal M}
\sup_{|\alpha| \leq N}\sup_{ \xi_1+\xi_2+\xi_3=0}\big| \xi_1-\xi_2 \big|^\alpha \big| \partial_\alpha m ({\xi})\big| \lesssim_N 1.
\end{equation}
The bilinear Hilbert transform \cite{LT1,LT2} corresponds to the (formal) choice $m(\xi)=\mathrm{sign}(\xi_1-\xi_2)$. Sparse bounds for this type of operators were first established, and fully characterized in the open range, in \cite{CuDPOu}, where it was proved that\begin{equation}
\label{BHTsparse}
\sup_{m \in \mathcal M}   \|T_m \|_{\vec p} <\infty
\iff
 1<p_1,p_2,p_3<\infty, \quad \sum_{j=1}^3\frac{1}{\min\{p_j,2\}}  <2.
 \end{equation}
Therefore, Theorem \ref{weighted:thm} with $n=2$ may be applied for any $\vec p$ in the range \eqref{BHTsparse}.  It is easy to see that
there exists  such a $\vec p$ with $1\leq q_1\leq p_1,1\leq q_2\leq p_2$ for all $(q_1,q_2)$ belonging to the sharp open range of unweighted strong-type estimates for the multipliers $\{T_m:m\in \mathcal M\}$, namely
\begin{equation}
\label{BHTscal}
1<q_1,q_2 \leq \infty, \qquad \frac 23 <q<\infty.
\end{equation}
 Therefore,  Theorem \ref{weighted:thm}, together with  its version for $q>1$ described in Remark \ref{qm1}, yield weighted, vector-valued boundedness of the multipliers $\{T_m:m\in \mathcal M\}$  for weights $v_1,v_2$ satisfying conditions 2.\ and 3.\  and the exponents recover the full unweighted range.
 
  Weighted bounds in such a full range, under more stringent assumption on the weights were obtained in \cite{CUM} by extrapolation of the results of \cite{CuDPOu}. The vector-valued analogue of the results in \cite{CUM} was instead proved  in \cite{BM2} by making use of vector-valued sparse bounds in a different way. To illustrate the subtle difference between the class of weights allowed in \cite{BM2,CUM} and those falling within the scope of Theorem \ref{weighted:thm}, we particularize our result to the diagonal case $q_1=q_2=2q$ with $\frac{2}{3}<q<\infty$. This is done  for simplicity of description of the multilinear classes $A_{(q_1,\ldots,q_n)}^{(t_1,\ldots, t_{n+1})}$ when $t_{n+1}=1,t_1=\cdots=t_n$, but off diagonal results can also be obtained in a similar fashion.

Note that the tuple (parametrized by $s$)
\[
p_1=p_2=\frac{2}{s}, \quad p_3=\frac{1}{2-s}+\delta, \qquad 1\leq s\leq \frac32 
\]
satisfies the conditions in \eqref{BHTsparse} for all $\delta>0$. As noted  in \cite[Lemma 3.2]{Chaffee2017}, if $qs\geq 1$, then
\begin{equation}
\label{RC1}
(v_1,v_2) \in A_{(2q,2q)}^{(\frac{2}{s},\frac{2}{s},1)} \iff   v_1,v_2 \in \mathrm{RC}\left( \frac{1}{1-qs}, \frac{1}{1+qs} \right) \supsetneq   A_{qs}, \,v=(v_1v_2)^{\frac12}\in A_{2qs}.
\end{equation}
Recall from \cite{IMS} that for $-\infty\leq \alpha< \beta\leq \infty$, the weight class $\mathrm{RC}(\alpha, \beta)$ contains those weights $w$ on $\R^d$ such that  
\[
\l w\r_{\beta, Q} \leq C  \l w\r_{\alpha, Q},
\]
with $C$ uniform over all cubes $Q$ of $\R^d$.  In particular, for $1\leq t<\infty$
\[
A_t = \mathrm{RC}\ \left( \frac{1}{1-t}, 1 \right), \qquad {RH}_t= \mathrm{RC}\left( 1, t \right).
\]
and the strict inclusion in \eqref{RC1} follows from the obvious relations   $
\alpha\leq \gamma \leq \delta\leq \beta\implies \mathrm{RC}(\alpha, \beta) \subset \mathrm{RC}(\gamma, \delta ) $. This  observation characterizes the weights that will verify the second assumption of Theorem \ref{weighted:thm}. Finally, rewriting the third assumption for our choice of tuple $\vec p$ yields the following result, which strictly contains   the diagonal case of the main results of  \cite{CUM}  (see also \cite{BM2} for the vector-valued analogue). 
\begin{theorem} \label{weighted:thm2}
Let $\frac23<q\leq 1$, $v_1,v_2$ be weights on $\R$. Assume that there exist
\[
s\in \textstyle \left[\frac1q, \frac{3}{2} \right]  , \qquad t \in \left[1, \frac{1}{s-1} \right)
\] such that
\[
v_1,v_2 \in  \mathrm{RC}\left( \frac{1}{1-qs}, \frac{1}{1+qs} \right) \supsetneq A_{qs}
\]
and
\[
v:=(v_1v_2)^{\frac12}\in A_{\min\{t,2qs\}}  \cap RH_{\frac{1}{1-t(s-1)}}.
\] 
Then the vector-valued strong type bound
\begin{equation}
\label{stbBHT}
\cic{T}=\{T_{m_j}:m_j \in \mathcal M\}: \prod_{j=1}^2 L^{2q}(v_j ; \ell^{r_j} ) \to  L^{q}(v; \ell^{r})  
\end{equation}
  holds true whenever $\min\{r_1,r_2\}\geq\frac{2}{s},  r_{3}= r'\geq\frac{1}{2-s}$.
 \end{theorem} 
For instance, the estimate, valid for all vector-valued tuples with $\min\{r_1,r_2,r_3\} \geq 2$, 
\[
\cic{T}: \prod_{j=1}^2 L^{2q}(v; \ell^{r_j} ) \to  L^{q}(v; \ell^{r}) , \qquad v_1,v_2 \in  A_{\frac{3q}{ 2}},  v \in A_{\frac{3q}{ 2}}  \cap RH_2,\quad \frac23<q\leq 1, \quad  
\]
follows by taking $s =\frac32,t=1$ in Theorem \ref{weighted:thm2}. This result includes    \cite[Corollary 4]{CuDPOu}, in vector-valued form.

\subsection*{Acknowledgment}
The authors would like to thank Kangwei Li for fruitful discussions on the weak-type weighted theory of multisublinear maximal functions.  
\bibliography{FeffermanStein}

\providecommand{\bysame}{\leavevmode\hbox to3em{\hrulefill}\thinspace}
\providecommand{\MR}{\relax\ifhmode\unskip\space\fi MR }
\providecommand{\MRhref}[2]{%
  \href{http://www.ams.org/mathscinet-getitem?mr=#1}{#2}
}
\providecommand{\href}[2]{#2}
\begin{thebibliography}{10}

\bibitem{Bar2017}
Alexander {Barron}, \emph{{Weighted Estimates for Rough Bilinear Singular
  Integrals via Sparse Domination}}, New York J. Math \textbf{23} (2017),
  779--811.

\bibitem{BM2}
C.~{Benea} and C.~{Muscalu}, \emph{{Sparse domination via the helicoidal
  method}}, preprint arXiv:1707.05484 (2017).

\bibitem{BFP}
Fr\'ed\'eric Bernicot, Dorothee Frey, and Stefanie Petermichl, \emph{Sharp
  weighted norm estimates beyond {C}alder\'on-{Z}ygmund theory}, Anal. PDE
  \textbf{9} (2016), no.~5, 1079--1113. \MR{3531367}

\bibitem{Chaffee2017}
Lucas Chaffee, Rodolfo~H. Torres, and Xinfeng Wu, \emph{Multilinear weighted
  norm inequalities under integral type regularity conditions}, Harmonic
  Analysis, Partial Differential Equations and Applications: In Honor of
  Richard L. Wheeden, eds. Chanillo et.\ al., Springer International Publishing
  (2017), 193--216.

\bibitem{CoCuDPOu}
Jos\'e~M. Conde-Alonso, Amalia Culiuc, Francesco Di~Plinio, and Yumeng Ou,
  \emph{A sparse domination principle for rough singular integrals}, Anal. PDE
  \textbf{10} (2017), no.~5, 1255--1284. \MR{3668591}

\bibitem{CUM}
D.~{Cruz-Uribe} and J.~{Mar{\'{\i}}a Martell}, \emph{{Limited range multilinear
  extrapolation with applications to the bilinear Hilbert transform}}, preprint
  arXiv:1704.06833 (2017).

\bibitem{CuDPOu}
Amalia Culiuc, Francesco Di~Plinio, and Yumeng Ou, \emph{Domination of
  multilinear singular integrals by positive sparse forms}, preprint
  arXiv:1603.05317.

\bibitem{DPDU}
F.~{Di Plinio}, Y.~Q. {Do}, and G.~N. {Uraltsev}, \emph{{Positive sparse
  domination of variational Carleson operators}}, arXiv:1612.03028, to appear
  on Ann.\ Scuola Norm.\ Sup.\ (Scienze) (2016).

\bibitem{DPHL}
F.~{Di Plinio}, T.~P. {Hyt{\"o}nen}, and K.~{Li}, \emph{{Sparse bounds for
  maximal rough singular integrals via the Fourier transform}}, preprint
  arXiv:1706.09064 (2017).

\bibitem{IMS}
Sapto Indratno, Diego Maldonado, and Sharad Silwal, \emph{A visual formalism
  for weights satisfying reverse inequalities}, Expo. Math. \textbf{33} (2015),
  no.~1, 1--29. \MR{3310925}

\bibitem{Lac2017}
M.~T. {Lacey}, \emph{{Sparse Bounds for Spherical Maximal Functions}}, preprint
  arXiv:1702.08594, to appear on J.\ d'Analyse Math. (2017).

\bibitem{LT1}
Michael Lacey and Christoph Thiele, \emph{{$L^p$} estimates on the bilinear
  {H}ilbert transform for {$2<p<\infty$}}, Ann. of Math. (2) \textbf{146}
  (1997), no.~3, 693--724. \MR{1491450 (99b:42014)}

\bibitem{LT2}
\bysame, \emph{On {C}alder\'on's conjecture}, Ann. of Math. (2) \textbf{149}
  (1999), no.~2, 475--496. \MR{1689336 (2000d:42003)}

\bibitem{Lac2015}
Michael~T. Lacey, \emph{An elementary proof of the {$A_2$} bound}, Israel J.
  Math. \textbf{217} (2017), no.~1, 181--195. \MR{3625108}

\bibitem{LMena}
Michael~T. Lacey and Dario Mena~Arias, \emph{The sparse {T}1 {Theorem}},
  Houston J.\ Math. \textbf{43} (2017), no.~1, 111--127.

\bibitem{Ler2013}
Andrei~K. Lerner, \emph{A simple proof of the {$A_2$} conjecture}, Int. Math.
  Res. Not. IMRN (2013), no.~14, 3159--3170. \MR{3085756}

\bibitem{Ler2016}
\bysame, \emph{On pointwise estimates involving sparse operators}, New York J.
  Math. \textbf{22} (2016), 341--349. \MR{3484688}

\bibitem{LOP+}
Andrei~K. Lerner, Sheldy Ombrosi, Carlos P\'erez, Rodolfo~H. Torres, and
  Rodrigo Trujillo-Gonz\'alez, \emph{New maximal functions and multiple weights
  for the multilinear {C}alder\'on-{Z}ygmund theory}, Adv. Math. \textbf{220}
  (2009), no.~4, 1222--1264. \MR{2483720}

\bibitem{Stein}
Elias~M. Stein, \emph{Harmonic analysis: real-variable methods, orthogonality,
  and oscillatory integrals}, Princeton Mathematical Series, vol.~43, Princeton
  University Press, Princeton, NJ, 1993, With the assistance of Timothy S.
  Murphy, Monographs in Harmonic Analysis, III. \MR{1232192}

\end{thebibliography}
\bibliographystyle{amsplain}
\end{document}